\documentclass{amsproc}
\usepackage[margin=1.2in,nomarginpar]{geometry}
\usepackage{amssymb}
\usepackage{amsmath}
\usepackage{mathdots}
\usepackage{amsbsy}
\usepackage{amscd}
\usepackage{amsthm}

\usepackage{url}
\usepackage{xcolor}
\usepackage{amsmath,amssymb,amsthm,amsfonts,longtable}
\usepackage[english]{babel}
\usepackage{tikz-cd}
\usetikzlibrary{cd}
\usetikzlibrary{decorations.markings}
\tikzset{negated/.style={
		decoration={markings,
			mark= at position 0.5 with {
				\node[transform shape] (tempnode) {$\times$};
			}
		},
		postaction={decorate}
	}
}
\usepackage{array}
\usepackage[colorlinks,citecolor=red,urlcolor=blue,bookmarks=false,hypertexnames=true]{hyperref} 
\usepackage{multirow}
\usepackage{pbox}
\newtheorem{theorem}{Theorem}
\newtheorem{corollary}[theorem]{Corollary}
\newtheorem{lemma}[theorem]{Lemma}

\newtheorem{remark}[theorem]{Remark}

\newcommand{\Irr}{\textnormal{Irr}}

\newcommand{\cd}{\textnormal{cd}}
\newcommand{\nl}{\textnormal{nl}}
\newcommand{\lin}{\textnormal{lin}}

\newcommand{\gal}{\textnormal{Gal}}

\title[]{Rational Group Algebras of Camina $p$-groups}
\author{Ram Karan Choudhary}
\address{Indian Institute of Technology, Bhubaneswar, Arugul Campus, Jatni, Khurda-752050, India.}
\email{ramkchoudhary1997@gmail.com}
\author{Sunil Kumar Prajapati$^*$}
\address{Indian Institute of Technology, Bhubaneswar, Arugul Campus, Jatni, Khurda-752050, India.}
\email{skprajapati@iitbbs.ac.in}
\thanks{$^{\textbf{*}}$ Corresponding author.
}
\begin{document}
	
	\begin{abstract}
		In this article, we present a combinatorial formula for the Wedderburn decomposition of rational group algebras of Camina $p$-groups, where $p$ is a prime. We also provide a complete set of primitive central idempotents of rational group algebras of these groups.
	\end{abstract}
	\maketitle

	\section{Introduction} 
	All groups considered in this paper are finite. Throughout this paper, let $p$ denote a prime number. For a finite group $G$, we use $Z(G)$, $G'$ and $\mathbb{Q}G$ to represent its center, derived subgroup and rational group algebra, respectively. Additionally, the notation $\Irr(G)$, $\lin(G)$, $\nl(G)$ and $\cd(G)$ correspond to the sets of irreducible complex characters, linear complex characters, non-linear irreducible complex characters and degrees of irreducible complex representations of $G$, respectively. Given $\chi \in \Irr(G)$, we denote the Schur index of $\chi$ over $\mathbb{Q}$ by $m_{\mathbb{Q}}(\chi)$, and define $\Omega(\chi) = m_{\mathbb{Q}}(\chi) \sum_{\sigma \in \text{Gal}(\mathbb{Q}(\chi) / \mathbb{Q})} \chi^{\sigma}$, where $\mathbb{Q}(\chi)$ represents the field obtained by adjoining to $\mathbb{Q}$ the character values $\{\chi(g) : g \in G\}$. Moreover, we use $M_q(D)$ to denote a full matrix ring of order $q$ over the division ring $D$, $Z(B)$ for the center of an algebraic structure $B$ and $\zeta_d$ for a primitive $d$-th root of unity. This paper focuses on the Wedderburn decomposition of rational group algebras of a particular class of finite $p$-groups $G$, characterized by the property that for each $g \in G \setminus G'$, we have $gG' \subseteq Cl_G(g)$, where $Cl_G(g)$ denotes the conjugacy class of $g$ in $G$. This class of finite $p$-groups, initially introduced by Camina \cite{Camina}, has been extensively studied in~\cite{Dark, Isaac-Lewis, MLL3, Macdonald, Macdonald2, Mann 1} and is known as \textit{Camina $p$-groups}.
	
	The Wedderburn-Artin theorem states that a ring is semisimple if and only if it decomposes into a direct sum of matrix rings over division rings. Furthermore, by the Brauer-Witt theorem (see \cite{Yam}), the Wedderburn components of a rational group algebra are Brauer equivalent to cyclotomic algebras. The study of Wedderburn decomposition of rational group algebras has attracted considerable attention due to its relevance in understanding various algebraic structures (see \cite{Herman, Jes-Rio, Rit-Seh}), and has been investigated using Shoda pair theory in many articles including \cite{BM14, BGO, Jes-Lea-Paq, Jes-Olt-Rio, Olt07}. Perlis and Walker \cite{PW} provided a combinatorial formula for the Wedderburn decomposition of rational group algebras of finite abelian groups based on counting their cyclic subgroups. In \cite{Ram}, we established a similar combinatorial description for the Wedderburn decomposition of rational group algebras of VZ $p$-groups (A group $G$ is said to be VZ group if $\chi(g)=0$ for all $\chi\in \nl(G)$ and for all $g\in G\setminus Z(G)$). This work extends those results by deriving combinatorial formulas for the Wedderburn decomposition of rational group algebras of Camina $p$-groups. Note that Dark and Scoppola \cite{Dark} demonstrated that for a finite Camina $p$-group, the nilpotency class is at most $3$. We prove Theorem \ref{thm:WedderburnCamina}, which provides the Wedderburn decomposition of rational group algebras of Camina $p$-groups.
	
	\begin{theorem}\label{thm:WedderburnCamina}
		Let $G$ be a Camina $p$-group of nilpotency class $r$, where $p$ is a prime. Then we have the following.
		\begin{enumerate}
			\item For $r=2$, we have the following two cases.
			\begin{enumerate}
				\item {\bf Case ($p\neq2$).} In this case, the Wedderburn decomposition of $\mathbb{Q}G$ is given by
				$$\mathbb{Q}G \cong \mathbb{Q} \bigoplus \frac{|G/Z(G)|-1}{p-1}\mathbb{Q}(\zeta_p) \bigoplus \frac{|Z(G)|-1}{p-1}M_{|G/Z(G)|^{\frac{1}{2}}}(\mathbb{Q}(\zeta_p)).$$ 
				
				\item {\bf Case ($p = 2$).} In this case, the Wedderburn decomposition of $\mathbb{Q}G$ is given by
				$$\mathbb{Q}G \cong |G/Z(G)|\mathbb{Q} \bigoplus kM_{\frac{1}{2}|G/Z(G)|^{\frac{1}{2}}}(\mathbb{H}(\mathbb{Q})) \bigoplus (|Z(G)|-k-1) M_{|G/Z(G)|^{\frac{1}{2}}}(\mathbb{Q}),$$ 
				where $k=|\{\chi\in \nl(G) : m_{\mathbb{Q}}(\chi)=2\}|$ and $\mathbb{H}(\mathbb{Q})$ is the standard quaternion algebra over $\mathbb{Q}$.
			\end{enumerate}
			
			\item For $r=3$, the Wedderburn decomposition of $\mathbb{Q}G$ is given by
			$$\mathbb{Q}G \cong \mathbb{Q} \bigoplus \frac{|G/G'|-1}{p-1}\mathbb{Q}(\zeta_p) \bigoplus \frac{|G'/Z(G)|-1}{p-1}M_{|G'/Z(G)|}(\mathbb{Q}(\zeta_p)) \bigoplus \frac{|Z(G)|-1}{p-1} M_{|G/Z(G)|^{\frac{1}{2}}}(\mathbb{Q}(\zeta_p)).$$  
		\end{enumerate}
	\end{theorem}
	
   Next, we have Corollary \ref{cor:isoCamina}, which directly follows from Theorem \ref{thm:WedderburnCamina}.
	\begin{corollary}\label{cor:isoCamina}
		Let $G$ and $H$ be two isoclinic, non-abelian Camina $p$-groups of the same order, where $p$ is a prime. Then $\mathbb{Q}G \cong \mathbb{Q}H$.
	\end{corollary}
	In this article, we also present a concise analysis of primitive central idempotents and their associated simple components in the Wedderburn decomposition of rational group algebras of Camina $p$-groups. Specifically, we prove Theorem \ref{thm:pciCamina}.
	\begin{theorem}\label{thm:pciCamina}
		Let $G$ be a Camina $p$-group of nilpotency class $r$, where $p$ is a prime. Let $\chi \in \nl(G)$ with $N = \ker(\chi)$. Then the following hold.
		\begin{enumerate}
			\item For $r=2$, we have the following.
			\begin{enumerate}
				\item $e_{\mathbb{Q}}(\chi)= \epsilon(Z(G), N)$. 
				
				\item If $p$ is odd, then $\mathbb{Q}G\epsilon(Z(G), N) \cong M_{|G/Z(G)|^{\frac{1}{2}}}(\mathbb{Q}(\zeta_{p}))$.
				
				\item If $p=2$, then				
					\begin{equation*}
			\mathbb{Q}G\epsilon(Z(G), N) \cong \begin{cases}
				M_{|G/Z(G)|^{\frac{1}{2}}}(\mathbb{Q}) & \text{when } m_\mathbb{Q}(\chi)=1, \\
			M_{\frac{1}{2}|G/Z(G)|^{\frac{1}{2}}}(\mathbb{H}(\mathbb{Q}))& \text{when}~  m_\mathbb{Q}(\chi)=2.
			\end{cases}
		\end{equation*}	
			\end{enumerate}
			\item For $r=3$, we have the following.
			\begin{enumerate}
				\item If $Z(G) \subseteq N$, then $e_{\mathbb{Q}}(\chi)= \epsilon(G', N)$ and $\mathbb{Q}G\epsilon(G', N) \cong M_{|G'/Z(G)|}(\mathbb{Q}(\zeta_{p}))$.
				\item If $Z(G) \nsubseteq N$, then $e_{\mathbb{Q}}(\chi)= \epsilon(Z(G), N)$ and $\mathbb{Q}G\epsilon(Z(G), N) \cong M_{|G/Z(G)|^{\frac{1}{2}}}(\mathbb{Q}(\zeta_{p}))$.
			\end{enumerate}
		\end{enumerate} 
	\end{theorem}
	The organization of the article is as follows. In Section \ref{sec:preliminaries}, we present essential preliminary results that will be used in the subsequent sections to prove our main results. Section \ref{sec:RationalGroupAlgebras} describes rational group algebras of Camina $p$-groups and contains the proof of Theorem \ref{thm:WedderburnCamina}. Finally, Section \ref{sec:pci} outlines the structure of primitive central idempotents in rational group algebras of Camina $p$-groups and provides the proof of Theorem \ref{thm:pciCamina}.

	\section{Preliminaries}\label{sec:preliminaries}
	In this section, we introduce some fundamental concepts and results that will be utilized frequently throughout the article. Let $G$ be a finite group. The nilpotency class of $G$ is defined as the smallest positive integer $n$ such that $G_n \neq 1$ and $G_{n+1} = 1$, where $G_2 = [G, G] = G'$ and $G_{i+1} = [G_i, G]$ for $i > 2$. For a normal subgroup $N$ of $G$, we denote the set of all irreducible complex characters of $G$ whose kernels do not contain $N$ by $\Irr(G|N)$, i.e.,
	$$\Irr(G|N) = \{\chi \in \Irr(G) : N \not\subseteq \ker(\chi)\}.$$
	We now recall some essential results that will be required later. The pair $(G, N)$ is termed a \textit{Camina pair} if for every $g \in G \setminus N$, the coset $gN$ is contained in the conjugacy class $Cl_G(g)$ of $g$. In the special case where $N = G'$, the group $G$ itself is called a \textit{Camina group}. The following lemma provides equivalent conditions for a pair $(G, N)$ to be a Camina pair.
	
	\begin{lemma}\cite[Lemma 3]{Mattarei}\label{lemma:prop of Camina pair}
		Let $N$ be a normal subgroup of $G$, and let $g \in G \setminus N$. Then the following statements are equivalent.
		\begin{enumerate}
			\item $\chi(g) = 0$ for all $\chi \in \Irr(G | N)$.
			\item $gN \subseteq Cl_G(g)$.
		\end{enumerate}
	\end{lemma}
	It is straightforward to verify that if $(G, N)$ is a Camina pair, then $Z(G) \subseteq N \subseteq G'$. Moreover, by Lemma \ref{lemma:prop of Camina pair}, it follows that if $G$ is a Camina group, then for all $g \in G \setminus G'$, we have $\chi(g) = 0$ for every $\chi \in \Irr(G | G')$.	
%
%
%
	We now describe some group-theoretic properties of Camina $p$-groups in Lemmas \ref{lemma:Caminaprop1}, \ref{lemma:Caminaprop2} and \ref{lemma:Caminaprop3}.
	
	\begin{lemma}\cite[Corollary 2.3]{Macdonald}\label{lemma:Caminaprop1}
		Let $G$ be a $p$-group of nilpotency class $r$. If $(G, G_k)$ forms a Camina pair, then $G_i/G_{i+1}$ has exponent $p$ for all $k-1 \leq i \leq r$.
	\end{lemma}
	
	\begin{lemma}\cite[Theorem 5.2]{Macdonald}\label{lemma:Caminaprop2}
		Let $G$ be a Camina $p$-group of nilpotency class $3$, with $|G/G_2| = p^m$ and $|G_2/G_3| = p^n$. Then:
		\begin{enumerate}
			\item $(G, G_3)$ is a Camina pair.
			\item $m = 2n$ and $n$ is even.
		\end{enumerate}
	\end{lemma}
	
	\begin{lemma}\cite[Corollary 5.3]{Macdonald}\label{lemma:Caminaprop3}
		If $G$ is a Camina $p$-group of nilpotency class $3$, then $Z_2(G) = G_2$ and $Z(G) = G_3$, where $Z_2(G)/Z(G) = Z(G/Z(G))$.
	\end{lemma}
	
	Lewis \cite{MLL3} initiated the study of groups for which $(G, Z(G))$ forms a Camina pair and established that such a group must be a $p$-group for some prime $p$. The following lemma describes a relationship between $\Irr(G | Z(G))$ and $\Irr(Z(G))$ when $(G, Z(G))$ is a Camina pair.
	
	\begin{lemma}\cite[Lemma 3.3]{SKP}\label{lemma:Caminacharacter}
		Let $G$ be a finite group, and $(G, Z(G))$ be a Camina pair. Then there exists a bijection between the sets $\Irr(G | Z(G))$ and $\Irr(Z(G)) \setminus \{1_{Z(G)}\}$, where $1_{Z(G)}$ is the trivial character of $Z(G)$. For $\mu \in \Irr(Z(G)) \setminus \{1_{Z(G)}\}$, the corresponding character $\chi_\mu \in \nl(G)$ is given by:
		\begin{equation}\label{Caminacharacter}
			\chi_\mu(g) = \begin{cases}
				|G/Z(G)|^{1/2} \mu(g) & \text{if } g \in Z(G), \\
				0 & \text{otherwise}.
			\end{cases}
		\end{equation}
	\end{lemma}
	
	Next, we recall some results regarding the Wedderburn decomposition of rational group algebras associated with some classes of finite groups that will be used in subsequent sections to prove our main results. In \cite{PW}, Perlis and Walker studied the group ring of a finite abelian group $G$ over the field of rational numbers and proved the following result.
	\begin{lemma}[Perlis-Walker Theorem]\label{Perlis-walker}
		Let $G$ be a finite abelian group of exponent $m$. Then the Wedderburn decomposition of $\mathbb{Q}G$ is given by
		\[\mathbb{Q}G \cong \bigoplus_{d|m} a_d \mathbb{Q}(\zeta_d),\]
		where $a_d$ is equal to the number of cyclic subgroups of $G$ of order $d$.
	\end{lemma} 
	
\noindent In \cite{Ram}, we formulate the computation of the Wedderburn decomposition of rational group algebra of a VZ $p$-group $G$ solely based on computing the number of cyclic subgroups of $G/G'$, $Z(G)$ and $Z(G)/G'$, which is similar to the Perlis-Walker theorem for an abelian group.
		\begin{lemma}\cite[Theorem 1]{Ram}\label{lemma:Wedderburn VZ}
		Let $G$ be a finite VZ $p$-group, where $p$ is an odd prime. Let $m_1$, $m_2$ and $m_3$ denote the exponents of $G/G'$, $Z(G)$ and $Z(G)/G'$, respectively. Then the Wedderburn decomposition of $\mathbb{Q}G$ is given by
		$$\mathbb{Q}G \cong \bigoplus_{d_1|m_1}a_{d_1}\mathbb{Q}(\zeta_{d_1})\bigoplus_{d_2\mid m_2, d_2 \nmid m_3} a_{d_2}M_{|G/Z(G)|^{\frac{1}{2}}}(\mathbb{Q}(\zeta_{d_2}))\bigoplus_{d_2|m_2, d_2|m_3}(a_{d_2}-a_{d_2}')M_{|G/Z(G)|^{\frac{1}{2}}}(\mathbb{Q}(\zeta_{d_2})),$$
		where $a_{d_1}$, $a_{d_2}$ and $a_{d_2}'$ are the number of cyclic subgroups of $G/G'$ of order $d_1$, the number of cyclic subgroups of $Z(G)$ of order $d_2$ and the number of cyclic subgroups of $Z(G)/G'$ of order $d_2$, respectively.
	\end{lemma}
	
	\begin{lemma}\label{lemma:WedderburnVZ 2-gp}
		Let $G$ be a VZ $2$-group. Let $m_1$, $m_2$ and $m_3$ denote the exponents of $G/G'$, $Z(G)$ and $Z(G)/G'$, respectively. Suppose $k=|\{\chi\in \nl(G) : m_{\mathbb{Q}}(\chi)=2\}|$, and $\mathbb{H}(\mathbb{Q})$ represents the standard quaternion algebra over $\mathbb{Q}$. Then the Wedderburn decomposition of $\mathbb{Q}G$ is given by
		\begin{align*}
			\mathbb{Q}G \cong &\bigoplus_{d_1\mid m_1}a_{d_1}\mathbb{Q}(\zeta_{d_1}) \bigoplus kM_{\frac{1}{2}|G/Z(G)|^{\frac{1}{2}}}(\mathbb{H}(\mathbb{Q})) \bigoplus(a_2-a_2'-k)M_{|G/Z(G)|^{\frac{1}{2}}}(\mathbb{Q})\\
			&\bigoplus_{d_2\mid m_2, d_2 \nmid m_3} a_{d_2}M_{|G/Z(G)|^{\frac{1}{2}}}(\mathbb{Q}(\zeta_{d_2})) \bigoplus_{d_2\mid m_2, d_2 \mid m_3}(a_{d_2}-a_{d_2}')M_{|G/Z(G)|^{\frac{1}{2}}}(\mathbb{Q}(\zeta_{d_2})),
		\end{align*}
		where $a_{d_1}$, $a_{l}$ and $a_{l}'$ $(l \in \{2, d_2\geq 4\})$ are the number of cyclic subgroups of $G/G'$ of order $d_1$, the number of cyclic subgroups of $Z(G)$ of order $l$ and the number of cyclic subgroups of $Z(G)/G'$ of order $l$, respectively.
	\end{lemma} 
	We close this section by quoting some well-known results that will be also used in the upcoming sections. 
	
	\begin{lemma}\cite[Corollary 10.14]{I}\label{lemma:schurindexpgroup}
		Let $G$ be a $p$-group, and let $\chi \in \Irr(G)$. If $p$ is an odd prime, then $m_\mathbb{Q}(\chi) = 1$; otherwise, $m_\mathbb{Q}(\chi) \in \{1,2\}$.
	\end{lemma}
	
	Let $G$ be a finite group. We define an equivalence relation on $\Irr(G)$ by Galois conjugacy over $\mathbb{Q}$. Two characters $\chi, \psi \in \Irr(G)$ are said to be {\it Galois conjugates} over $\mathbb{Q}$ if $\mathbb{Q}(\chi) = \mathbb{Q}(\psi)$ and there exists $\sigma \in \gal(\mathbb{Q}(\chi) / \mathbb{Q})$ such that $\chi^\sigma = \psi$.
	
	\begin{lemma}\textnormal{\cite[Lemma 9.17]{I}}\label{SC}
		Suppose $G$ is a finite group and $\chi \in \Irr(G)$. Let $E(\chi)$ denote the Galois conjugacy class of a complex irreducible character $\chi$ over $\mathbb{Q}$. Then
		\[ |E(\chi)| = [\mathbb{Q}(\chi) : \mathbb{Q}]. \]
	\end{lemma}
	
	Note that the distinct Galois conjugacy classes correspond to the distinct irreducible rational representations of $G$. Reiner \cite[Theorem 3]{IR} further classified the structure of the simple algebra appearing in the Wedderburn decomposition of the rational group algebra of a finite group $G$ associated with an irreducible rational representation of $G$. Finally, we quote Reiner's result.

	\begin{lemma}\cite[Theorem 3]{IR} \label{Reiner}
		Let $\mathbb{K}$ be an arbitrary field with characteristic zero and  $\mathbb{K}^*$ be the algebraic closure of $\mathbb{K}$. Suppose $T$ is an irreducible $\mathbb{K}$-representation of $G$, and extend $T$ (by linearity) to a $\mathbb{K}$-representation of $\mathbb{K}G$. Set
		\[A=\{T(x) : x\in \mathbb{K}G\}.\]
		Then $A$ is a simple algebra over $\mathbb{K}$, and we may write  $A= M_{n}(D)$, where $D$ is a division ring. Further,
		\begin{center}
			$Z(D)\cong \mathbb{K}(\chi_i)$ and $[D : Z(D)]=(m_\mathbb{K}(\chi_i))^2~(1\leq i \leq k),$
		\end{center}
		where $U_i$ are irreducible $K^{*}$-representations of $G$ that affords the character $\chi_i$, $T=m_{\mathbb{K}}(\chi_i)\bigoplus_{i=1}^k U_i$, $k=[\mathbb{K}(\chi_i) : \mathbb{K}]$ and $m_{\mathbb{K}}(\chi_i)$ is the Schur index of $\chi_i$ over $\mathbb{K}$.
	\end{lemma}

	\section{Rational group algebra}\label{sec:RationalGroupAlgebras}
	In this section, we present the proof of Theorem \ref{thm:WedderburnCamina}. Before that, we prove some essential results required for its proof. Let $G$ be a finite group, and let $\chi, \psi \in \Irr(G)$ such that $\chi$ and $\psi$ are Galois conjugates over $\mathbb{Q}$. Then we observe that $\ker(\chi) = \ker(\psi)$. The converse also holds when $\chi, \psi \in \lin(G)$.	
In general, if $\chi, \psi \in \nl(G)$ satisfy $\ker(\chi) = \ker(\psi)$, then $\chi$ may not necessarily be Galois conjugate to $\psi$ over $\mathbb{Q}$. When $(G,Z(G))$ forms a Camina pair, we have an analogous result for certain non-linear irreducible characters of $G$.	
	\begin{lemma}\label{lem:galoisCaminaCenter}
		Let $G$ be a finite group such that $(G, Z(G))$ forms a Camina pair. Suppose $\chi, \psi \in \Irr(G | Z(G))$. Then $\chi$ and $\psi$ are Galois conjugates over $\mathbb{Q}$ if and only if $\ker(\chi) = \ker(\psi)$.
	\end{lemma} 
	
	\begin{proof}
		Suppose $\chi, \psi \in \Irr(G | Z(G))$ are Galois conjugates over $\mathbb{Q}$. It follows immediately that $\ker(\chi) = \ker(\psi)$. 
		Conversely, assume $\chi, \psi \in \Irr(G | Z(G))$ satisfy $\ker(\chi) = \ker(\psi)$. By \eqref{Caminacharacter}, there exist $\mu, \nu \in \Irr(Z(G)) \setminus \{1_{Z(G)}\}$ such that $\chi \downarrow_{Z(G)} = \mu$ and $\psi \downarrow_{Z(G)} = \nu$. Since $\ker(\chi) = \ker(\mu)$ and $\ker(\psi) = \ker(\nu)$, linear characters $\mu$ and $\nu$ are Galois conjugates. Consequently, $\mathbb{Q}(\mu) = \mathbb{Q}(\nu)$, implying the existence of $\sigma \in \operatorname{Gal}(\mathbb{Q}(\mu) / \mathbb{Q})$ such that $\mu^\sigma = \nu$. Furthermore, since $\mathbb{Q}(\chi) = \mathbb{Q}(\mu)$ and $\mathbb{Q}(\psi) = \mathbb{Q}(\nu)$, it follows that $\chi^\sigma = \psi$. Thus, $\chi$ and $\psi$ are Galois conjugates over $\mathbb{Q}$. This completes the proof.
	\end{proof}
	
	Now, suppose $G$ is a finite group and $\chi, \psi \in \Irr(G)$. We say that $\chi$ and $\psi$ are \textit{equivalent} if $\ker(\chi) = \ker(\psi)$.
	
	\begin{lemma}\label{lemma:Ayoub}\cite[Lemma 1]{Ayoub}
		Let $G$ be a finite abelian group of exponent $m$, and let $d$ divide $m$. Suppose $a_d$ denotes the number of cyclic subgroups of $G$ of order $d$. Then the number of inequivalent characters $\chi$ satisfying $\mathbb{Q}(\chi) = \mathbb{Q}(\zeta_d)$ is $a_d$.
	\end{lemma}
	
	Analogous to Lemma \ref{lemma:Ayoub}, we establish the following result.
	
	\begin{lemma}\label{lemma:CaminaCentergaloisconjugates}
		Let $G$ be a finite group such that $(G, Z(G))$ forms a Camina pair. Suppose $d \neq 1$ is a divisor of $\exp(Z(G))$, and let $a_d$ denote the number of cyclic subgroups of $Z(G)$ of order $d$. Then the number of inequivalent characters $\chi \in \Irr(G | Z(G))$ satisfying $\mathbb{Q}(\chi) = \mathbb{Q}(\zeta_d)$ is $a_d$.
	\end{lemma}
	
	\begin{proof}
		Let $\chi \in \Irr(G | Z(G))$. Then there exists $\mu \in \Irr(Z(G)) \setminus \{1_{Z(G)}\}$ such that $\chi = \chi_\mu$ as given in \eqref{Caminacharacter}. Observe that $\ker(\chi_\mu) = \ker(\mu)$ and $\mathbb{Q}(\chi_\mu) = \mathbb{Q}(\mu)$. Furthermore, Lemma \ref{lemma:Caminacharacter} establishes a bijection between the sets $\Irr(G | Z(G))$ and $\Irr(Z(G)) \setminus \{1_{Z(G)}\}$. Consequently, by Lemma \ref{lemma:Ayoub}, the number of inequivalent characters $\chi \in \Irr(G | Z(G))$ satisfying $\mathbb{Q}(\chi) = \mathbb{Q}(\zeta_d)$ is precisely $a_d$. This completes the proof of Lemma \ref{lemma:CaminaCentergaloisconjugates}.
	\end{proof}
		
	Now, we proceed to the proof of Theorem \ref{thm:WedderburnCamina}.
		\begin{proof}[Proof of Theorem \ref{thm:WedderburnCamina}]
		Let $G$ be a Camina $p$-group of nilpotency class $r$, where $r\in \{2,3\}$ and $p$ is a prime.
		
		\begin{enumerate}
			\item Suppose $r=2$. By Lemma \ref{lemma:Caminaprop1}, both the derived subgroup $G'$ and the quotient $G/G'$ are elementary abelian $p$-groups. Since $G$ is a Camina group, every nonlinear irreducible character of $G$ vanishes on $G \setminus G'$, and since nilpotency class of $G$ is $2$, we get $G{}'\subseteq Z(G)$. This implies that $G' = Z(G)$, and hence $G$ is a VZ $p$-group. The desired results then follow from Lemmas \ref{lemma:Wedderburn VZ} and \ref{lemma:WedderburnVZ 2-gp}. This concludes the proof of Theorem \ref{thm:WedderburnCamina}(1).
			
			\item Now, consider the case when $r=3$. From Lemma \ref{lemma:Caminaprop3}, we have $Z(G) \subseteq G'$ and $G'/Z(G)$ is abelian. Furthermore, by Lemmas \ref{lemma:Caminaprop1} and \ref{lemma:Caminaprop2}, the groups $G/G'$, $G'/Z(G)$ and $Z(G)$ are all elementary abelian $p$-groups.\\
			Let $\chi \in \Irr(G)$. Suppose $\rho$ is an irreducible $\mathbb{Q}$-representation of $G$ affording the character $\Omega(\chi)$. Let $A_\mathbb{Q}(\chi)$ be the simple component in the Wedderburn decomposition of $\mathbb{Q}G$ corresponding to $\rho$, which is isomorphic to $M_q(D)$ for some $q \in \mathbb{N}$ and a division ring $D$. Since $G$ is a Camina $p$-group of nilpotency class $3$, we apply \cite[Theorem 3.1]{Macdonald2} to conclude that $p$ is an odd prime. Consequently, by Lemma \ref{lemma:schurindexpgroup}, we have $m_\mathbb{Q}(\chi) = 1$. Moreover, Lemma \ref{Reiner} ensures that $[D: Z(D)] = m_\mathbb{Q}(\chi)^2$ and $Z(D) = \mathbb{Q}(\chi)$. Thus, it follows that $D = Z(D) = \mathbb{Q}(\chi)$. Next, consider $\rho = \bigoplus_{i=1}^{l} \rho_i$, where $l = [\mathbb{Q}(\chi): \mathbb{Q}]$ and each $\rho_i$ is an irreducible complex representation of $G$ affording the character $\chi^{\sigma_i}$ for some $\sigma_i \in \mathrm{Gal}(\mathbb{Q}(\chi):\mathbb{Q})$. Since we have $m_\mathbb{Q}(\chi) = 1$, it follows from \cite[Theorem 3.3.1]{JR} that $q = \chi(1)$. \\
			From Lemmas \ref{lemma:Caminaprop2} and \ref{lemma:Caminaprop3}, the pair $(G, Z(G))$ is a Camina pair. For any $\mu \in \Irr(Z(G)) \setminus \{1_{Z(G)}\}$, let $\chi_\mu \in \Irr(G|Z(G))$ be defined as in \eqref{Caminacharacter}. Then we have $\chi_\mu(1) = |G/Z(G)|^{\frac{1}{2}}$ and $\mathbb{Q}(\chi_\mu) = \mathbb{Q}(\mu)$. Let $\rho$ be an irreducible $\mathbb{Q}$-representation of $G$ affording the character $\Omega(\chi_\mu)$. Then, as discussed earlier, we have $A_\mathbb{Q}(\chi_\mu) \cong M_{|G/Z(G)|^{\frac{1}{2}}}(\mathbb{Q}(\mu))$. Observe that $\mathbb{Q}(\chi_\mu) = \mathbb{Q}(\mu) = \mathbb{Q}(\zeta_{d})$, for some $d \mid \exp(Z(G))$. Moreover, $Z(G)$ is an elementary abelian $p$-group. Therefore, from Lemmas \ref{lem:galoisCaminaCenter} and \ref{lemma:CaminaCentergaloisconjugates}, the simple components of the Wedderburn decomposition of $\mathbb{Q}G$ corresponding to all those irreducible $\mathbb{Q}$-representations of $G$ whose kernels do not contain $Z(G)$ are
			$$\bigoplus \frac{|Z(G)|-1}{p-1} M_{|G/Z(G)|^{\frac{1}{2}}}(\mathbb{Q}(\zeta_p)).$$
			Further, by Lemma \ref{lemma:Caminaprop2}, we have $|G/G'|=p^{2n}$, $|G'/Z(G)|=p^n$ and $|G/Z(G)|=p^{3n}$ for some even integer $n$. Since $(G, G')$ is a Camina pair and $Z(G) \leq G' \leq G$, it follows that $(G/Z(G), G'/Z(G))$ is also a Camina pair. By Lemma \ref{lemma:Caminaprop3}, we have $Z(G/Z(G))=G'/Z(G)=(G/Z(G))'$, implying that $G/Z(G)$ is a Camina $p$-group of nilpotency class $2$, where $p$ is an odd prime. For $\chi \in \Irr(G) \setminus \Irr(G|Z(G))$, there exists $\bar{\chi} \in \Irr(G/Z(G))$ such that $\bar{\chi}$ lifts to $\chi$. Moreover, there is a bijection between the sets $\Irr(G) \setminus \Irr(G|Z(G))$ and $\Irr(G/Z(G))$. Note that $\bar{\chi}(1)=[G/Z(G): Z(G/Z(G))]^{\frac{1}{2}}=[G/Z(G): G{}'/Z(G)]^{\frac{1}{2}}=p^n=|G{}'/Z(G)|$. Therefore, from Theorem \ref{thm:WedderburnCamina}(1), the simple components of the Wedderburn decomposition of $\mathbb{Q}G$ corresponding to all those irreducible $\mathbb{Q}$-representations of $G$ whose kernels contain $Z(G)$ are
			$$ \mathbb{Q} \bigoplus \frac{|G/G'|-1}{p-1}\mathbb{Q}(\zeta_p) \bigoplus \frac{|G'/Z(G)|-1}{p-1}M_{|G'/Z(G)|}(\mathbb{Q}(\zeta_p)).$$ 
			By combining these results, we conclude the proof of Theorem \ref{thm:WedderburnCamina}(2).
		\end{enumerate}
		This completes the proof of Theorem \ref{thm:WedderburnCamina}.
	\end{proof}
	
	Next, we present the proof of Corollary \ref{cor:isoCamina}.
	\begin{proof}[Proof of  Corollary \ref{cor:isoCamina}]
		Suppose that $G$ and $H$ are two isoclinic, non-abelian Camina $p$-groups of the same order, where $p$ is a prime. By the definition of isoclinism, we have $G' \cong H'$ and $G/Z(G) \cong H/Z(H)$. Consequently, applying Theorem \ref{thm:WedderburnCamina} yields the desired result. This concludes the proof of Corollary \ref{cor:isoCamina}.
	\end{proof}

\begin{remark}\label{rem:extraspecial} \textnormal{It is a well known fact that for each positive integer $n$ there are exactly two (up to isomorphism) extraspecial $p$-groups of order $p^{2n+1}$ and they are isoclinic. Moreover, it is easy to verify that extraspecial $p$-groups are Camina $p$-groups of nilpotency class $2$. Hence, from Corollary \ref{cor:isoCamina}, their rational group algebras are isomorphic. Further, from Theorem \ref{thm:WedderburnCamina}, if $G$ is an extraspecial $p$-group with an odd prime $p$, then  
$$\mathbb{Q}G \cong \mathbb{Q} \oplus (p^{2n-1}+p^{2n-2}+\cdots+p+1)\mathbb{Q}(\zeta_p) \oplus M_{p^n}(\mathbb{Q}(\zeta_p)).$$}
\end{remark}

 Note that if two groups are isoclinc, then their rational group algebras need not be isomorphic. We conclude this section with the following Corollary which provides an example of a isoclinic family of non-extraspecial $p$-groups with isomorphic rational group algebras. 
	\begin{corollary}\label{cor:p-gpisogpalg}
		Let $G$ be a non-abelian $p$-group of order $p^6$ such that $G$ belongs to the isoclinic family $\Phi_{15}$ (see \cite{NewmanO`Brien}), where $p$ is an odd prime. Then 
			$$\mathbb{Q}G \cong \mathbb{Q} \oplus (p^3+p^2+p+1)\mathbb{Q}(\zeta_p) \oplus (p+1)M_{p^2}(\mathbb{Q}(\zeta_p)).$$
	\end{corollary}
	\begin{proof}
Let $G$ be a group of order $p^6$ such that $G \in \Phi_{15}$. Then $G'=Z(G)\cong C_p\times C_p$ and $G/Z(G)\cong C_p \times C_p \times C_p \times C_p$ (see \cite{NewmanO`Brien}). Moreover, $G$ is a Camina $p$-group of nilpotency class $2$. Hence, the result follows from Theorem \ref{thm:WedderburnCamina}. 
	\end{proof}

	\section{Primitive central idempotents}\label{sec:pci}
	In this section, we first provide a brief introduction to primitive central idempotents in rational group algebras, followed by the proof of Theorem \ref{thm:pciCamina}. Let $G$ be a finite group. An element $e$ in $\mathbb{Q}G$ is called an idempotent if it satisfies $e^2 = e$. A primitive central idempotent $e$ in $\mathbb{Q}G$ is an idempotent that lies in the center of $\mathbb{Q}G$ and cannot be expressed as a sum of two nonzero orthogonal idempotents, i.e., $e = e' + e''$ with $e'e'' = 0$.
	
	It is well known that a complete set of primitive central idempotents of $\mathbb{Q}G$ determines its decomposition into a direct sum of simple subalgebras. Specifically, if $e$ is a primitive central idempotent of $\mathbb{Q}G$, then the corresponding simple component of $\mathbb{Q}G$ is given by $\mathbb{Q}Ge$. For a character $\chi \in \Irr(G)$, the idempotent
	\begin{equation*}
		e(\chi) := \frac{\chi(1)}{|G|} \sum_{g \in G} \chi(g) g^{-1}
	\end{equation*}
	defines a primitive central idempotent in $\mathbb{C}G$. Furthermore, the set $\{e(\chi) : \chi \in \Irr(G)\}$ forms a complete set of primitive central idempotents of $\mathbb{C}G$. Additionally, for $\chi \in \Irr(G)$, we define
	\begin{equation*}
		e_{\mathbb{Q}}(\chi) := \sum_{\sigma \in \text{Gal}(\mathbb{Q}(\chi) / \mathbb{Q})} e(\chi^{\sigma}).
	\end{equation*}
	This yields a primitive central idempotent in $\mathbb{Q}G$.
	
	For a subset $X$ of $G$, define
	\begin{equation*}
		\widehat{X} := \frac{1}{|X|} \sum_{x \in X} x \in \mathbb{Q}G.
	\end{equation*}
	For a normal subgroup $N$ of $G$, define
	\begin{equation*}
		\epsilon(G, N) := \begin{cases}
			\widehat{G}, & \text{if } G = N; \\
			\prod_{D/N \in M(G/N)} (\widehat{N} - \widehat{D}), & \text{otherwise},
		\end{cases}
	\end{equation*}
	where $M(G/N)$ denotes the set of minimal nontrivial normal subgroups $D/N$ of $G/N$, with $D$ being a subgroup of $G$ containing $N$.
	
	Here, we compute a complete set of primitive central idempotents of the rational group algebra of a Camina $p$-group. We begin with a general result.
	
	\begin{lemma}\cite[Lemma 3.3.2]{JR}\label{lemma:pcilin}
		Let $G$ be a finite group, and let $\chi \in \lin(G)$ with $N=\ker(\chi)$. Then we have the following.
		\begin{enumerate}
			\item $e_{\mathbb{Q}}(\chi)= \epsilon(G, N)$.
			\item $\mathbb{Q}G\epsilon(G, N) \cong \mathbb{Q}(\zeta_{|G/N|})$. 
		\end{enumerate}
	\end{lemma}
	
	Now, we are ready to prove Theorem \ref{thm:pciCamina}.
	\begin{proof}[Proof of Theorem \ref{thm:pciCamina}]
		Let $G$ be a Camina $p$-group of nilpotency class $r$, where $r\in \{2,3\}$ and $p$ is a prime. Let $\chi \in \nl(G)$ with $N=\ker(\chi)$.
		
		\begin{enumerate}
			\item Assume that $G$ is a Camina $p$-group of nilpotency class $2$. In this case, $G$ is a VZ $p$-group. Hence, there exists $\mu \in \Irr(Z(G)|G')$ such that
			\begin{equation*}
				\chi = \chi_\mu(g) = \begin{cases}
					|G/Z(G)|^{1/2} \mu(g) & \text{if } g \in Z(G), \\
					0 & \text{otherwise.}
				\end{cases}
			\end{equation*}
			Moreover, we note that $e(\chi) = e(\chi_\mu) = e(\mu)$. Thus, we have
			\begin{align*}
				e_{\mathbb{Q}}(\chi) &= e_{\mathbb{Q}}(\chi_\mu) \\
				&= \sum_{\sigma \in \operatorname{Gal}(\mathbb{Q}(\chi_\mu) / \mathbb{Q})} e(\chi_\mu^{\sigma}) \\
				&= \sum_{\sigma \in \operatorname{Gal}(\mathbb{Q}(\mu) / \mathbb{Q})} e(\chi_{\mu^{\sigma}}) \\
				&= \sum_{\sigma \in \operatorname{Gal}(\mathbb{Q}(\mu) / \mathbb{Q})} e(\mu^{\sigma}) \\
				&= e_{\mathbb{Q}}(\mu) \\
				&= \epsilon(Z(G), N),
			\end{align*}
			where $N = \ker(\mu) = \ker(\chi_\mu) = \ker(\chi)$ (see Lemma \ref{lemma:pcilin}).\\
			If $p$ is an odd prime, Theorem \ref{thm:WedderburnCamina} states that the simple subalgebra of $\mathbb{Q}G$ generated by $e_{\mathbb{Q}}(\chi)$ is given by
			\begin{equation*}
				\mathbb{Q}G e_{\mathbb{Q}}(\chi) = \mathbb{Q}G\epsilon(Z(G), N) \cong M_{|G/Z(G)|^{1/2}}(\mathbb{Q}(\zeta_p)).
			\end{equation*}
			For $p=2$ and $m_\mathbb{Q}(\chi) = 1$, Theorem \ref{thm:WedderburnCamina} yields
			\begin{equation*}
				\mathbb{Q}G e_{\mathbb{Q}}(\chi) = \mathbb{Q}G\epsilon(Z(G), N) \cong M_{|G/Z(G)|^{1/2}}(\mathbb{Q}(\zeta_p)).
			\end{equation*}
			When $p=2$ and $m_\mathbb{Q}(\chi) = 2$, Theorem \ref{thm:WedderburnCamina} implies that
			\begin{equation*}
				\mathbb{Q}G e_{\mathbb{Q}}(\chi) = \mathbb{Q}G\epsilon(Z(G), N) \cong M_{(1/2)|G/Z(G)|^{1/2}}(\mathbb{H}(\mathbb{Q})).
			\end{equation*}
			
			\item Now, suppose $G$ is a Camina $p$-group of nilpotency class $3$. Then $p$ is necessarily an odd prime (see \cite[Theorem 3.1]{Macdonald2}). Consider $\chi \in \nl(G)$ such that $Z(G) \subseteq N$. Let $\bar{\chi} \in \Irr(G/Z(G))$ be the character that lifts to $\chi$. Since $G/Z(G)$ is a Camina $p$-group of nilpotency class $2$ and $Z(G/Z(G)) = G'/Z(G)$ (see Lemma \ref{lemma:Caminaprop3}), it follows that
			\begin{align*}
				e_{\mathbb{Q}}(\bar{\chi}) &= \epsilon(Z(G/Z(G)), N/Z(G)) \\
				&= \epsilon(G'/Z(G), N/Z(G)),
			\end{align*}
	where $\ker(\bar{\chi})=N/Z(G)$. Further, as $\chi$ is the lift of $\bar{\chi}$ to $G$, and hence from the above discussion, we deduce that $e_{\mathbb{Q}}(\chi) = \epsilon(G', N)$. Moreover, by Theorem \ref{thm:WedderburnCamina}, the simple subalgebra of $\mathbb{Q}G$ generated by $e_{\mathbb{Q}}(\chi)$ satisfies
			\begin{equation*}
				\mathbb{Q}G e_{\mathbb{Q}}(\chi) = \mathbb{Q}G\epsilon(G', N) \cong M_{|G'/Z(G)|}(\mathbb{Q}(\zeta_p)).
			\end{equation*}
			Lastly, consider $\chi \in \nl(G)$ such that $Z(G) \nsubseteq N$. Then there exists a unique $\mu \in \Irr(Z(G)) \setminus \{1_{Z(G)}\}$ such that
			\begin{equation*}
				\chi_\mu(g) = \begin{cases}
					|G/Z(G)|^{1/2} \mu(g) & \text{if } g \in Z(G), \\
					0 & \text{otherwise,}
				\end{cases}
			\end{equation*}
			(see \eqref{Caminacharacter}). The conclusions then follow analogously from the proof of Theorem \ref{thm:pciCamina}(1) and Theorem \ref{thm:WedderburnCamina}.
		\end{enumerate}
		This completes the proof of Theorem \ref{thm:pciCamina}.
	\end{proof}

	\section{Acknowledgments}
	Ram Karan acknowledges University Grants Commission, Government of India.

\end{document}